\documentclass[a4paper]{amsart}
\usepackage{url}
\usepackage{amsmath}
\usepackage[makeroom]{cancel}
\usepackage{graphicx}
\usepackage[all]{xy}
\usepackage[mathscr]{eucal}
\usepackage{amsmath,amssymb,amsfonts}

\newtheorem{theorem}{Theorem}[section]
\newtheorem{lemma}[theorem]{Lemma}
\newtheorem{corollary}[theorem]{Corollary}
\newtheorem{example}[theorem]{Example}

\newtheorem{proposition}[theorem]{Proposition}
\newtheorem{remark}[theorem]{Remark}

\usepackage{stackengine}
\usepackage{xcolor}
\usepackage[all]{xy}

\title[An adjunction theorem]{An adjunction between a category of complete restriction monoids and a category of \'etale categories}

\author{Mark V. Lawson}
\address{Mark V. Lawson, Department of Mathematics
and the
Maxwell Institute for Mathematical Sciences,
Heriot-Watt University,
Riccarton,
Edinburgh EH14 4AS,
UNITED KINGDOM}
\email{m.v.lawson@hw.ac.uk}

\begin{document}

\begin{abstract}
We prove that there is an adjunction between what we call \'etale topological categories
and restriction quantal frames that leads to an adjunction with a category of complete restriction monoids.
This generalizes the adjunction between \'etale groupoids and pseudogroups.
\end{abstract}
\maketitle

\section{Introduction}

What do we mean by the phrase `noncommutative topological space'?
One answer, of course, is provided by $C^{\ast}$-algebras \`{a} la Connes \cite{Connes}.
A more concrete answer is provided by \'etale groupoids \cite{Paterson, Renault, Resende1}.
An explicit description of \'etale groupoids as noncommutative topological spaces can be found in \cite[1. Introduction]{K}.
However, groupoids come equipped with an involution which renders them rather special kinds of noncommutative spaces.
On the other hand, the more general \'etale categories do not share that feature.
We therefore regard \'etale categories as being noncommutative topological spaces.
This approach was pursued in \cite{KL} where we developed the theory of \'etale categories from the localic point of view inspired by Resende's work \cite{Resende2}. 
I would like to thank my co-author, Ganna Kudryavtseva, who particularly pushed for this point of view.
In this paper, on the other hand, the theory of \'etale categories is developed from the {\em topological point of view}
extending the approach described in \cite{LL} for \'etale groupoids,
though we should stress here that the proofs have to be different.
We shall prove, as Theorem~\ref{them:adjunction-redux}, that there is an adjunction between
a suitable category of complete restriction monoids and a suitable catgeory of (small) \'etale categories.  
The motivation for this paper came from \cite{Bice} and \cite{CM}.

\section{Notation and background results}

Let $(X,\leq)$ be a poset.
If $Y \subseteq X$ define 
$$Y^{\uparrow} = \{x \in X \colon \exists y \in X;  y \leq x\}$$ 
and
$$Y^{\downarrow} = \{x \in X \colon \exists y \in X;  x \leq y\}.$$
In the case where $Y = \{y\}$, we write $y^{\uparrow}$ and $y^{\downarrow}$ 
instead of $\{y\}^{\uparrow}$ and $\{y\}^{\downarrow}$, respectively. 
If $Y = Y^{\downarrow}$ we say that $Y$ is an {\em order-ideal}.
The subset $Y$ is said to be {\em downwards directed} if $x,y \in Y$ implies that
there is $z \in Y$ such that $z \leq x,y$.
If $Y = Y^{\uparrow}$ we say that $Y$ is {\em closed upwards}.
The {\em bottom element} of $X$, if it exists, is unique and is the smallest element of $X$;
we shall denote it by $0$.
If $X$ and $Y$ are posets then a function $f \colon X \rightarrow Y$ is {\em order preserving}
if $x \leq x'$ implies that $f(x) \leq f(x')$.
A bijective function $f$ is said to be an {\em order isomorphism} if both $f$ and $f^{-1}$ are order preserving.

We shall be working with {\em small} categories $C$.
In such a category, all elements will be arrows.
Amongst such arrows are the identities.
The set of identities of $C$ is denoted by $C_{o}$.
Let $a \in C$.
Denote by $\mathbf{d}(a)$ the unique right identity of $a$, and denote by $\mathbf{r}(a)$
the unique left identity of $a$.
Observe that $ab$ is defined in the category $C$ if and only if $\mathbf{d}(a) = \mathbf{r}(b)$.
We write $\exists ab$ if the product $ab$ is defined in the category.
Denote by $C \ast C$ the set of all ordered pairs $(a,b)$ such that  $\mathbf{d}(a) = \mathbf{r}(b)$.
Denote by $\mathbf{m} \colon C \ast C \rightarrow C$ the multiplication map.
We therefore regard categories as `monoids with many objects'.
Let $C$ be a small category.
A subset $A \subseteq C$ is said to be a {\em local bisection} if $\mathbf{d}$ (respectively, $\mathbf{r}$) is injective when restricted to $A$.
You can check that the product of local bisections is again a local bisection.
A functor $F \colon C \rightarrow D$ between categories is said to be
{\em $\mathbf{d}$-injective}
if $F(x) = F(y)$ and $\mathbf{d}(x) = \mathbf{d}(y)$ imply that $x = y$;
such a functor is said to be  
{\em $\mathbf{d}$-surjective} 
if $F(e)= \mathbf{d}(y)$
implies there exists an element $x \in C$ such that $\mathbf{d}(x) = e$ and $F(x) = y$.
A functor which is both $\mathbf{d}$-injective and $\mathbf{d}$-surjective 
is said to be  $\mathbf{d}$-bijective.
We may make similar definitions for the $\mathbf{r}$-function. 
A {\em covering} functor is a functor that is both $\mathbf{d}$-bijective and $\mathbf{r}$-bijective.
The proof of the following is immediate from the definitions.

\begin{lemma}\label{lem:covering-functors-identities} 
Let $F \colon C \rightarrow D$ be a $\mathbf{d}$-injective functor (such as a covering functor). 
Then if $e$ is an identity in $D$ and $F(a) = e$ then $a$ is an identity.
\end{lemma}

A category $C$ is said to be {\em topological} if $C$ is equipped with a topology in such a way that
the maps $\mathbf{m}$, $\mathbf{d}$ and $\mathbf{r}$ are all continuous.

We conclude by saying something about the classical theory of topological spaces and frames.
This is discussed in more detail in \cite[Chapter II]{J}.
There are two ways to regard a space:
we can take the notion of point as primary and derive the notion of open subset as secondary (an open set is a special set of points), 
or we can take the notion of open set as primary and the notion of point as secondary (how to get points from open sets will be discussed later);
the former leads to the classical notion of topological space whereas the latter leads to the notions of frames and locales --- the so-called `pointless spaces'.
There is a relationship between these two approaches but it is not an isomorphism.
A {\em frame} is a join-complete lattice in which finite meets distribute over arbitrary joins.
Frames have top elements ($1$) and bottom elements ($0$) and morphisms of frames are required to preserve these elements,
arbitrary joins and finite meets.
In this way, we obtain a category of frames which we shall denote by \mbox{\bf Frame}.
Its opposite category is called the category of locales denoted by \mbox{\bf Locale}.\\

\noindent
{\bf Definition. }If $X$ is a topological space, then its set of all open subsets $\Omega (X)$ forms a frame.\\

If $f \colon X \rightarrow Y$ is a continuous function then $f^{-1} \colon \Omega (Y) \rightarrow \Omega (X)$ is a morphism of frames.
Define $\Omega (f) = f^{-1}$.
Denote the category of topological spaces and continuous functions by \mbox{\bf Top}.
We have therefore defined a functor $\Omega$ from \mbox{\bf Top} to \mbox{\bf Locale}.

We shall define a functor going in the opposite direction, but before we do that
we shall motivate our definitions.\\

\noindent
{\bf Definition. }Let $X$ be a topological space and $x \in X$.
Define $O_{x}$ to be the set of all open sets that contain $x$.\\

Observe that $O_{x} \subseteq \Omega (X)$ is a completely prime filter \cite[Page 41]{J}.\\

\noindent
{\bf Definition. }If $F$ is a frame, denote by $\mathsf{pt}(F)$ the set of all completely prime filters in $F$;
we think of these as the `points' of $F$.\\

Be aware that in the frame $\Omega (X)$, 
it is quite possible for different points to give rise to the same completely prime filter, 
and we are not claiming that every completely prime filter arises from a point.
Nevertheless, this is the best we can do.
Given a frame $F$, we have a set of points $\mathsf{pt}(F)$.
We shall now endow this set with a topology.\\

\noindent
{\bf Definition. } Let $F$ be a frame.
For each $a \in F$, define $X_{a}$ to be the set of all completely prime filters in $F$ that contain $a$.\\

Observe that $X_{a} \subseteq \mathsf{pt}(F)$.
We have the following:
\begin{enumerate}
\item $X_{0} = \varnothing$.
\item $X_{1} = \mathsf{pt}(F)$.
\item $X_{a} \cap X_{b} = X_{a \wedge b}$.
\item $\bigcup_{i \in I} X_{a_{i}} = X_{\bigvee a_{i}}$.
\end{enumerate}
It follows that the set $\{X_{a} \colon a \in F\}$ forms a topology on $\mathsf{pt}(F)$.
If $\alpha \colon F \rightarrow G$ is a morphism of frames then $\alpha^{-1} \colon \mathsf{pt}(G) \rightarrow \mathsf{pt}(F)$
is a continous function.
Define $\mathsf{pt}(\alpha) = \alpha^{-1}$.
We have therefore defined a functor $\mathsf{pt}$ from \mbox{\bf Locale} to \mbox{\bf Top}.

By \cite[Theorem 1.4]{J}, the functor $\mathsf{pt}$ is right adjoint to the functor $\Omega$.
Define a function $\chi \colon F \rightarrow \Omega (\mathsf{pt}(F))$ by $a \mapsto X_{a}$.
This function is an isomorphism of frames if and only if whenever $a \nleq b$ there exists a completely prime filter
containing $a$ but omitting $b$.
In this case, we say that the frame is {\em spatial}.
Define a function $\omega \colon X \rightarrow \mathsf{pt}(\Omega (X))$ by $x \mapsto O_{x}$.
The function is a homeomorphism if and only if $\omega$ is a bijection.
In this case, we say that the topological space is {\em sober}.
The following were proved in \cite[page 43, Lemma 1.7]{J}.

\begin{lemma}\label{lem:sobriety-spatiality}\mbox{}
\begin{enumerate}
\item For any frame $F$, the topological space $\mathsf{pt}(F)$ is sober,
\item For any topological space $X$, the frame $\Omega (X)$ is spatial.
\end{enumerate}
\end{lemma}

The above adjunction leads to an equivalence of categories between spatial locales and sober spaces.

\section{Etale topological categories}

The class of semigroups we shall be studying are the following and were introduced in \cite{Lawson1991}.
Let $S$ be a semigroup with set of idempotents $\mathsf{E}(S)$.
We say that $S$ is an {\em Ehresmann semigroup} if there is a commutative subsemigroup $\mathsf{P} \subseteq \mathsf{E}(S)$,
called the set of {\em projections}, 
together with two unary operations $a \mapsto a^{\ast}$ and $a \mapsto a^{+}$, where $a^{\ast}, a^{+} \in \mathsf{P}$
and which are the identity on $\mathsf{P}$, such that $aa^{\ast} = a$, $a^{+}a = a$ and $(ab)^{\ast} = (a^{\ast}b)^{\ast}$
and $(ab)^{+} = (ab^{+})^{+}$.
There is a binary relation defined on Ehresmann semigroups as follows:
the ordered pair $(a,b)$ is in this binary relation if $a = a^{+}b = ba^{\ast}$.
In fact, this binary relation is a partial order.
It plays an important role in what follows but not for all elements.

Let $S$ and $T$ be Ehresmann semigroups.
A semigroup homorphism $\theta \colon S \rightarrow T$ is said to be a {\em morphism}
if $\theta (a^{\ast}) = \theta (a)^{\ast}$ and $\theta (a^{+}) = \theta (a)^{+}$.

With every Ehresmann semigroup $S$, we may associate a category $\mathsf{cat}(C)$ as follows:
define a partial product $\cdot$ on $S$ by $a \cdot b$ is defined if and only if $a^{\ast} = b^{+}$ in which case it is equal to $ab$;
the identities are the projections, where $\mathbf{d}(a) = a^{\ast}$ and $\mathbf{r}(b) = b^{+}$.
See \cite[page 437]{Lawson1991}.

An Ehresmann semigroup is called a {\em restriction} semigroup if it also satisfies two identities:
$fa = a(fa)^{\ast}$ and $af = (af)^{+}a$ for any projection $f$ and element $a$.
Classes of Ehresmann semigroups arose in the York School of semigroup theory \cite{GL}.
Restriction semigroups were studied by Marco Grandis in the context of local structures \cite{Grandis}.
Classes of restriction semigroups have been generalized to a categorical setting by Cockett et al
\cite{Cockett}.
The following example motivates our approach. 
See \cite[Example 1]{Lawson2021}.

\begin{example}\label{ex:example-one}
{\em Let $C$ be a small category.
For the time-being, we shall regard $C$ as a discrete topological space.
Consider the set $Q = \mathsf{P}(C)$ of all subsets of $C$.
If $A$ and $B$ are two such subsets then we may define their product $AB$ to be the set of all 
products $ab$ where $a \in A$, $b \in B$ and $\mathbf{d}(a) = \mathbf{r}(b)$.
In this way, $Q$ becomes a monoid with identity $C_{o}$.
However, the category structure of $C$ bleeds through to the monoid $Q$.
Put $\mathsf{P} = \mathsf{P}(C_{o})$.
These are a special class of idempotents that we shall call projections.
Define two unary operations on $Q$ as follows
$A^{\ast} = \{\mathbf{d}(a) \colon a \in A\}$
and
$A^{+}  = \{\mathbf{r}(a) \colon a \in A\}$.
With respect to these two operations, $Q$ becomes an Ehresmann monoid.}
\end{example}

We shall now describe the algebraic and order-theoretic properties of the monoids  $\mathsf{P}(C)$
introduced in Example~\ref{ex:example-one}.
Let $Q$ be a join-lattice equipped with a semigroup operation that distributes over joins.
Then $Q$ is called a {\em quantale}.
A {\em morphism} of quantales is a semigroup homomorphism that preserves joins.
Let $Q$ be a quantale.
We say that it is a {\em quantal frame} if $(Q,\leq)$ is a frame.
We say that a quantale is {\em unital} if the semigroup is actually a monoid.
An {\em Ehresmann quantale} is a unital quantale $Q$ that is also an Ehresmann semigroup with unit (or, identity) 
$e$ such that $e^{\downarrow} = \mathsf{P}$, and the maps $a \mapsto a^{\ast}$ and $a \mapsto a^{+}$ are join-maps.
An {\em Ehresmann quantal frame} is a quantal frame that is also an Ehresmann quantale;
these structures were introduced in \cite{KL}.
The Ehresmann monoids of the form $\mathsf{P}(C)$, as described in Example~\ref{ex:example-one},
are, in fact, Ehresmann quantal frames.

Let $C$ be a topological category.
We say that it is {\em\'etale} if both maps $\mathbf{d}$ and $\mathbf{r}$ are local homeomorphisms.
It follows \cite[Proposition 2.4]{Bice} that $C_{o}$ is an open subset of $C$ and the multiplication map is a local homeomorphism.
The following generalizes Example~\ref{ex:example-one}.

\begin{proposition}\label{prop:pelosi} Let $C$ be an \'etale topological category.
Then $\Omega (C)$ is an Ehresmann quantal frame.
\end{proposition}

We shall now describe an extra property possessed by monoids of the form $\Omega (C)$ as above.

\begin{lemma}\label{lem:base-base} Let $C$ be a topological categegory in which $\mathbf{d}$ and $\mathbf{r}$ are open maps.
Then the maps $\mathbf{d}$ and $\mathbf{r}$ are local homeomorphisms if
and only if the open local bisections form a base for the topology on $C$.
\end{lemma}
\begin{proof} Suppose first that the open local bisections form a base for the topology on $C$.
We prove that $\mathbf{d}$ is a local homeomorphism.
Let $x \in C$.
Then $x \in A$ for some open local bisection $A$.
The set $\mathbf{d}(A) = \{\mathbf{d}(a) \colon a \in A\}$ is open
and there is a bijection between $A$ and $\mathbf{d}(A)$.
We prove that this is a homeomorphism.
Let $B$ be an open subset of $A$.
Then $\mathbf{d}(B)$ is an open subset of $\mathbf{d}(A)$.
Let $U$ be an open subset of $\mathbf{d}(A)$.
Define $B$ to be the subset of $A$ consisting of elements $a \in A$ such that $\mathbf{d}(a) \in U$.
By assumption, the map $\mathbf{d}$ is continuous and so $B$ is also open.

We now prove the converse.
Suppose that $\mathbf{d}$ and $\mathbf{r}$ are local homeomorphisms.
We prove that the open local bisections form a base for the topology on $C$.
Let $x \in U$, where $U$ is any open subset of $C$.
By assumption, there is an open set $A_{x}$ such that 
$x \in A_{x}$ and the map $\mathbf{d}$ induces a homeomorphism onto the set $\mathbf{d}(A_{x})$.
Similarly, there is an open set $x \in B_{x}$ such that the map $\mathbf{r}$ induces a homeomorphism
onto the set $\mathbf{r}(B_{x})$.
Observe that $x \in U \cap A_{x} \cap B_{x} = A$ is an open set containing $x$
on which $\mathbf{d}$ and $\mathbf{r}$ are both bijections.
It follows that $A$ is an open local bisection containing $x$ and contained in $U$.
We have therefore proved that every open subset of $C$ is a union of open local bisections.
\end{proof}

By Lemma~\ref{lem:base-base}, it follows that the open local bisections form a base 
for the topology on an \'etale category. 

Let $S$ be an Ehresmann quantale.
An element $a$ in $S$ is called a {\em partial isometry} if $b \leq a$ implies that $b = b^{+}a = ab^{\ast}$.
Observe that all projections are partial isometries.
Partial isometries are those elements where the partial order in the quantale
is determined purely algebraically.\\

\noindent
{\bf Definition.} Let $S$ be a restriction quantal frame.
Denote the set of partial isometries in $S$ by $\mathsf{PI}(S)$.\\

Although not closed under multiplcation in general,
they do form an order-ideal \cite[Lemma 2.29]{KL};
this result will be used many times below.
The following result describes the partial isometries in \'etale categories.

\begin{lemma}\label{lem:partial-isometry} In an \'etale category, the open partial isometries are precisely the
open local bisections.
\end{lemma} 
\begin{proof} We prove first that every open local bisection is an open partial isometry.
Let $A$ be an open local bisection.
Suppose that $B \subseteq A$ where $A$ is an open set.
We prove that $B = AB^{\ast}$.
Clearly, $B \subseteq AB^{\ast}$.
We prove the reverse inclusion.
Let $ae \in AB^{\ast}$ where $a \in A$ and $e \in B^{\ast}$.
Then $\mathbf{d}(a) = e$ and $e = \mathbf{d}(b)$ where $b \in B$.
But $a,b \in A$ and $\mathbf{d}(a) = \mathbf{d}(b)$.
By assumption, $A$ is a local bisection.
It follows that $a = b$ and so $a \in B$.
A dual argument shows that $B = B^{+}A$.
We have therefore proved that $A \leq B$ and so $B$ is a partial isometry.

We now prove the converse.
We use Lemma~\ref{lem:base-base}.
Let $A$ be an open partial isometry.
We shall prove that $A$ is a local bisection.
Let $a,b \in A$ such that $\mathbf{d}(a) = \mathbf{d}(b) = e$.
Since $A$ is an open set, there exists an open local bisetcion $A_{a}$ such that $a \in A_{a} \subseteq A$.
But $A$ is a partial isometry, and so $A_{a} = AA_{a}^{\ast}$.
Similarly, there is an open local bisection $A_{b}$ such that $b \in A_{b} \subseteq A$.
Again, it follows that $A_{b} = AA_{b}^{\ast}$.
Observe that $a,b \in A_{b}$ which is a local bisection.
It follows that $a = b$.
A dual argument shows us that, indeed, $A$ is a local bisection.
\end{proof}

Let $Q$ be an Ehresmann quantal frame with identity $e$ and top element $1$.
We say that $Q$ is {\em \'etale} if $1$ is a join of partial isometries
and that it is a {\em restriction quantal frame} if it is \'etale and the partial isometries are closed under multiplication.
Such quantales were introduced in \cite[Section 2.4]{KL}.
We now have the following which is the basis of the whole paper; 
it follows by Lemma~\ref{lem:base-base} and Lemma~\ref{lem:partial-isometry} 
and the fact that the product of two open local bisections is an open local bisection.

\begin{proposition}\label{prop:kemi}
Let $C$ be an \'etale topological category.
Then $\Omega (C)$ is a restriction quantal frame.
\end{proposition}

We shall only deal with restriction quantal frames in what follows.
Let $S$ and $T$ be restriction quantal frames.
A {\em morphism} $\theta \colon S \rightarrow T$ 
of restriction quantal frames is a semigroup homomorphism that satisfies the following conditions:
\begin{enumerate}
\item $\theta$ preserves all joins.
\item $\theta$ preserves finite meets.
\item $\theta$ is a morphism of Ehresmann semigroups.
\item $\theta$ maps top elements to top elements, and is a monoid homomorphism.
\item $\theta$ maps partial isometries to partial isometries.
\end{enumerate}

\begin{lemma}\label{lem:maps} Let $\theta \colon C \rightarrow D$ be a continuous covering functor between \'etale categories.
Then $\theta^{-1}$ is a morphism $\Omega (D) \rightarrow \Omega (C)$ of restriction quantal frames.
\end{lemma}
\begin{proof} Because $\theta$ is continuous, it follows that $\theta^{-1}$ is a well-defined function.
We prove next that $\theta^{-1}$ is a homomorphism of semigroups.
Let $U,V \in \Omega (D)$.
To prove that $\theta^{-1}(UV) \subseteq \theta^{-1}(U)\theta^{-1}(V)$,
we use the fact that $\theta$ is a covering functor.
To prove the reverse inclusion, we simply use the fact that $\theta$ is a functor.
Observe that $\theta^{-1}(D_{o})$ contains $C_{o}$ but must equal this set because $\theta$ is a covering functor.
We have therefore proved that $\theta^{-1}$ is a monoid homomorphism.
We have proved that the partial isometries in a a restriction quantal frame constructed from  an \'etale category
are precisely the open local bisections.
Using the fact that $\theta$ is a continuous covering functor, it is routine to check that $\theta^{-1}$ maps
open local bisections to open local bisections.
Observe that $\theta^{-1}$ is automatically a frame map and so preserves joins and finite meets.
It is immediate that $\theta^{-1}$ maps the top element to the top element.
It remains to show that $\theta^{-1}$ is a morphism of Ehresmann semigroups.
Let $U$ be an open set of $D$.
We prove that $\theta^{-1}(U^{\ast}) = \theta^{-1}(U)^{\ast}$.
We therefore have to prove that $\theta^{-1}(\mathbf{d}(U)) = \mathbf{d}(\theta^{-1}(U))$.
The inclusion $\theta^{-1}(\mathbf{d}(U)) \subseteq \mathbf{d}(\theta^{-1}(U))$
follows from the fact that $\theta$ is a covering functor.
The reverse inclusion simply uses the fact that $\theta$ is a functor. 
\end{proof}

Define $\Omega (\theta) = \theta^{-1}$.
By Proposition~\ref{prop:kemi} and Lemma~\ref{lem:maps},
we have proved the following.

\begin{proposition}\label{prop:functor-one} 
$\Omega$ is a contravariant functor from the category whose objects are the \'etale topological categories 
and whose arrows are the continuous covering functors to the category whose objects are the restriction quantal frames
and whose arrows are the morphisms of restriction quantal frames.
\end{proposition}

The following example provides good motivation for what we are doing.

\begin{example}{\em Let $X$ be a non-empty finite set.
Then $X \times X$ becomes a category when we define $\mathbf{d}(x,y) = (y,y)$, $\mathbf{r}(x,y) = (x,x)$
and $(x,y)(y,z) = (x,z)$.
With respect to the discrete topology $X \times X$ is an \'etale category (in fact, it is an \'etale groupoid). 
The set of all subsets of $X \times X$ is $\mathsf{P}(X \times X)$ and is just the set of all binary relations on $X$.
The partial isometries are those binary relations induced by partial bijections.
It is now clear why the usual order is not algebraically defined unless we are dealing with elements which are `functional'
in nature.}
\end{example}

\section{Restriction quantal frames}

In the previous section, we showed how to construct restriction quantal frames from \'etale topological categories.
In this section, we go in the opposite direction;
we shall construct an \'etale topological category $\mathsf{C}(Q)$ from a restriction quantal frame $Q$.
Throughout this section, we shall assume that we are working with a restriction quantal frame.
The top element of the restriction quantal frame will be denoted by $1_{Q}$ or, simply, $1$.
The monoid identity will be denoted by $e_{Q}$ or, simply, $e$.
The partial order in a restriction quantal frame will be denoted by $\leq$. 

\begin{remark}
{\em 
Observe that in a restriction quantal frame $Q$, 
every element of $Q$ is a join of partial isometries (this is stated in the paragraph before \cite[Proposition 2.40]{KL}).
To see why,
we may write $1 = \bigvee_{i \in I} a_{i}$ where the $a_{i}$ are partial isometries.
Now, let $x$ be an arbitrary element of $Q$.
Then $x \leq 1$.
It follows that $x = \bigvee_{i \in I} a_{i} \wedge x$ since we are working in a frame.
We now use the fact that the partial isometries form an order-ideal.}
\end{remark}

Let $Q$ be a frame.
A {\em filter} in $Q$ is a non-empty subset that is closed under binary meets and is closed upwards.
Let $A$ be a set that is closed upwards and is downwards directed.
Then, in fact, $A$ is a filter.
To see why,
let $a,b \in A$ and $c \leq a,b$ where $c \in A$.
Then $c \leq a \wedge b$.
Since the set $A$ is closed upwards, it follows that $a \wedge b \in A$.
A {\em proper filter} is a filter that does not contain the bottom element.
A proper filter $A$ is said to be {\em completely prime}
if $\bigvee_{i \in I} a_{i} \in A$ implies that $a_{i} \in A$ for some $i \in I$. 
Completely prime filters in frames correspond to {\em points} \cite{J}.

\begin{remark}
{\em In checking that a subset of a quantal frame is a filter
it is enough to check that it is closed upwards and downwardly directed.}
\end{remark}

Let $Q$ be a restriction quantal frame.
If $a \vee b$ is a partial isometry then so too are $a$ and $b$ since the set of partial isometries forms 
an order-ideal. However, this necessary condition is not sufficient.
To obtain that, we need the notion of compatibility.
Define $a \sim b$ if and only if $ab^{\ast} = ba^{\ast}$ and $b^{+}a = a^{+}b$.
We call $\sim$ the {\em compatibility relation}.
If $a \sim b$ we say that $a$ and $b$ are {\em compatible}.

\begin{lemma}\label{lem:compatibility} Let $Q$ be a restriction quantal frame, 
and suppose that $a$ and $b$ are partial isometries.
Then $a \vee b$ is a partial isometry if and only if $a \sim b$.
\end{lemma}
\begin{proof} Suppose first that $a \sim b$. We prove that $a \vee b$ is a partial isometry.
Let $c \leq a \vee b$.
Then $c^{\ast} \leq a^{\ast} \vee b^{\ast}$.
Now, $ca^{\ast} \leq aa^{\ast} \vee ba^{\ast}$.
But $ba^{\ast} = ab^{\ast}$.
It follows that $ca^{\ast} \leq a$.
But $a$ is a partial isometry and so $ca^{\ast} = ac^{\ast}a^{\ast} = ac^{\ast}$.
Similarly, $cb^{\ast} = bc^{\ast}$.
Thus $c = c(a^{\ast} \vee b^{\ast}) = (a \vee b)c^{\ast}$.
The dual result delivers the fact that $a \vee b$ is a partial isometry.
We now prove the converse.
Suppose that $a \vee b$ is a partial isometry.
We  have that $a \leq a \vee b$.
Thus $a = a \vee ba^{\ast}$ and so $ba^{\ast} \leq a$.
But $a$ is a partial isometry.
It follows that $ba^{\ast} = ab^{\ast}$.
The dual result follows by symmetry and so $a$ and $b$ are compatible. 
\end{proof}

A restriction semigroup is said to be {\em complete} if every compatible subset has a join
and multiplication distributes over such joins.

\begin{proposition}\label{prop:pi-restriction-semigroup}
Let $Q$ be a restriction quantal frame.
Then the set of partial isometries of $Q$ forms a complete restriction monoid.
\end{proposition}
\begin{proof} Let $a$ be any element and $f$ a projection.
Then $fa \leq a$.
But $a$ is a partial isometry.
It follows that $fa = a(fa)^{\ast}$.
The dual result also holds.
This show that the set of partial isometries forms a restriction monoid.
It is now straightforward to check that it is in fact a complete restriction monoid.  
\end{proof}

\noindent
{\bf Definition. }Let $Q$ be a restriction quantal frame.
Denote the set of completely prime filters of $Q$ by $\mathsf{C}(Q)$.\\

We shall study this set.
We begin by relating completely prime filters in $Q$ to those in $e^{\downarrow}$.
Our first result shows how to map completely prime filters in  $e^{\downarrow}$
to completely prime filters in $Q$.

\begin{lemma}\label{lem:walpole} \mbox{}
\begin{enumerate}
\item Let $F$ be a completely prime filter in $e^{\downarrow}$.
Then $F^{\uparrow}$ is a completely prime filter in $Q$.
\item Let $F$ and $G$ be completely prime filters in  $e^{\downarrow}$.
Then $F^{\uparrow} = G^{\uparrow}$ if and only if $F = G$.
\end{enumerate}
\end{lemma} 
\begin{proof} (1) The only part of the proof that needs any comment is the following.
Suppose that $\bigvee_{i \in I} x_{i} \in F^{\uparrow}$.
Then there is a projection $f \in F$ such that $f \leq \bigvee_{i \in I} x_{i}$.
Thus, using the fact that we are working in a frame, we have that $f =  \bigvee_{i \in I} (x_{i} \wedge f)$.
But the set of projections is an order-ideal.
Thus each $x_{i} \wedge f$ is a projection.
It follows that $x_{i} \wedge f \in F$ for some $i$.
But $x_{i} \wedge f \leq x_{i}$ and so $x_{i} \in F^{\uparrow}$. 

(2) Only one direction needs proving. 
Suppose that $F^{\uparrow} = G^{\uparrow}$ 
Let $f \in F$.
Then $f \in F^{\uparrow}$ and so $f \in G^{\uparrow}$. 
There is therefore $g \in G$ such that $g \leq f$.
But both $f$ and $g$ are projections.
It follows that $f \in G$.
We have therefore proved that $F \subseteq G$.
The proof of the reverse inclusion follows by symmetry.
\end{proof}

The next result describes the completely prime filters that arise in $Q$ according to the procedure described in Lemma~\ref{lem:walpole}.

\begin{lemma}\label{lem:pitt} The completely prime filters in $Q$ which are of the form
$F^{\uparrow}$, for some completely prime filter $F$ in $e^{\downarrow}$, are the completely prime filters of $Q$ that contain projections.
\end{lemma}
\begin{proof} Let $A$ be a completely prime filter in $Q$ which contains a projection $f$.
Put $F = A \cap e^{\downarrow}$, a non-empty set by assumption.
Clearly, $F^{\uparrow} \subseteq A$.
We prove the reverse inclusion.
Let $a \in A$.
Then $a \wedge f \in A$.
But $a \wedge f$ is a projection since the set of projections forms an order-ideal.
On the other hand, $a \wedge f \leq a$.
It follows that $a \in F^{\uparrow}$.
It is now straightforward to check that $F$ is a completely prime filter in $e^{\downarrow}$.
\end{proof}

\begin{remark}
{\em Observe that in completely prime filters which contain projections,
every element is above a projection.}
\end{remark}

A completely prime filter of $Q$ that contains a projection will be called an {\em identity}.

\begin{lemma}\label{lem:thanos} Let $A$ be a completely prime filter in a restriction quantal frame $Q$.
Then, $A$ is an identity filter  if and only if $a^{\ast} \in A$ for any $a \in A$ (and, dually).
\end{lemma}
\begin{proof} Suppose that $A$ contains a projection $f$.
Let $a \in A$.
Then there is a projection $f' \in A$ such that $f' \leq a,f$.
It follows that $f' \leq a^{\ast}$.
We have therefore proved that $a^{\ast} \in A$.
The proof of the converse is immediate. 
\end{proof}

The following generalizes Lemma~\ref{lem:walpole} and Lemma~\ref{lem:pitt}.
It connects identity completely prime filters in $Q$ with completely prime filters in $e^{\downarrow}$.

\begin{proposition}\label{prop:bijection-cp} Let $Q$ be a restriction quantal frame.
Then there is an order isomorphism between the set of identity completely prime filters in $Q$
and the set of completely prime filters in $e^{\downarrow}$.
\end{proposition}
\begin{proof} Let $A$ be an identity completely prime filter of $Q$.
Then $A \cap e^{\downarrow}$ is a completely prime filter in $e^{\downarrow}$.
If $A \subseteq B$ then $A \cap e^{\downarrow} \subseteq B \cap e^{\downarrow}$.
On the other hand, if $F$ is a completely prime filter of  $e^{\downarrow}$
then $F^{\uparrow}$ is a completely prime filter of $Q$.
If $F \subseteq G$ then $F^{\uparrow} \subseteq G^{\uparrow}$.
Observe that $A = (A \cap e^{\downarrow})^{\uparrow}$
and $F = F^{\uparrow} \cap e^{\downarrow}$.
It follows that we have the order isomorphism as claimed.
\end{proof}

In our next result, we show two ways to construct completely prime filters in $e^{\downarrow}$ from completely prime filters in $Q$.
Let $A$ be a completely prime filter in $Q$.
Define 
$$A^{\ast} = \{a^{\ast} \colon a \in A\} \,\text{and}\, A^{+} = \{a^{+} \colon a \in A\}.$$

\begin{lemma}\label{lem:callghan} 
Let $A$ be a completely prime filter in $Q$. 
Then $A^{\ast}$ and $A^{+}$ are completely prime filters in $e^{\downarrow}$.
\end{lemma}
\begin{proof} We prove the statement for $A^{\ast}$ since the proof for $A^{+}$ is similar.
Observe that $0 \in A^{\ast}$ is and only if $a^{\ast} = 0$ for some $a \in A$ which implies that $a = 0$.
By assumption, the filter $A$ is proper and so $0 \notin A^{\ast}$.
Supppse that $a^{\ast} \subseteq f$ where $a \in A$ and $f$ is a projection.
Clearly, $a \subseteq 1$ thus $a \subseteq 1f$.
It follows that $1f \in A$.
Thus $(1f)^{\ast} = f \in A^{\ast}$.
We have proved that $A^{\ast}$ is closed upwards in $e^{\downarrow}$.
Let $a^{\ast},b^{\ast} \in A^{\ast}$ where $a,b \in A$.
By assumption, there exists $c \in A$ such that $c \leq a,b$.
But then $c^{\ast} \leq a^{\ast},b^{\ast}$.
We have therefore proved that $A^{\ast}$ is a proper filter in $e^{\downarrow}$.
We now prove that it is completely prime.
Let $\bigvee_{i \in I} f_{i} \in A^{\ast}$.
Then $a^{\ast} \leq \bigvee_{i \in I} f_{i}$ for some $a \in A$.
Clearly, $a \leq 1$.
It follows that $a \subseteq  \bigvee_{i \in I} 1f_{i}$.
We now use the fact that $A$ is a completely prime filter to deduce that $1f_{i} \in A$ for some $i \in I$.
It follows that $f_{i} \in A^{\ast}$.
\end{proof}

Define
$$\mathbf{d}(A) = (A^{\ast})^{\uparrow} \, \text{ and } \, \mathbf{r}(A) = (A^{+})^{\uparrow}.$$
By Lemma~\ref{lem:pitt} and Lemma~\ref{lem:callghan}, these are both completely prime filters in $Q$.
The proof of the following is straightforward.

\begin{lemma}\label{lem:heath-pm} Let $A$ and $B$ be completely prime filters.
Then $A^{\ast} = B^{+}$ if and only if $\mathbf{d}(A) = \mathbf{r}(B)$.
\end{lemma}

We now show how to construct all completely prime filters in $Q$ from those in $e^{\downarrow}$.

\begin{lemma}\label{lem:thatcher}\mbox{} 
\begin{enumerate}
\item Let $A$ be a completely prime filter with $a \in A$ arbitrary. Then $aA^{\ast} \subseteq A$, and dually.
\item Let $A$ be a completely prime filter. If $a \in A$ is a partial isometry then $A = (aA^{\ast})^{\uparrow} = (a \mathbf{d}(A))^{\uparrow}$, and dually.
\item Let $A$ be an identity completely prime filter in $Q$.
Suppose that $a^{\ast} \in A$ and that $a$ is a partial isometry.
Then $B = (aA)^{\uparrow}$ is a completely prime filter in $Q$ such that $\mathbf{d}(B) = A$.
\item Let $A$ and $B$ be completely prime filters such that $A$ and $B$ have a partial isometry in common and $\mathbf{d}(A) = \mathbf{d}(B)$.
Then $A = B$.
\end{enumerate}
\end{lemma}
\begin{proof} (1) Let $b \in B$ and let $c \in A$ be such that $c \subseteq a,b$.
Then $c = cc^{\ast} \subseteq ab^{\ast}$.
It follows that $ab^{\ast} \in A$.

(2) Let $x \in A$.
There exists $c \in A$ such that $c \leq x,a$.
But $a$ is a partial isometry and so $c = ac^{\ast}$.
Thus $ac^{\ast} \leq x$.
We have therefore proved that $A = (aA^{\ast})^{\uparrow}$.
It is easy to prove that $(aA^{\ast})^{\uparrow} = (a \mathbf{d}(A))^{\uparrow}$.

(3) It is easy to check that $B$ is a proper filter.
We prove that it is completely prime.
Suppose that $\bigvee x_{i} \in B$.
Then $au \leq  \bigvee x_{i}$ for some $u \in A$;
we may assume that $u$ is a projection.
Thus $au =  \bigvee (x_{i} \wedge au)$. 
Observe that $x_{i} \wedge au \leq au$.
Since $au$ is a partial isometry so too is $x_{i} \wedge au$.
It follows that $x_{i} \wedge au = au(x_{i} \wedge au)^{\ast}$ for each $i$.
Observe that $(x_{i} \wedge au)^{\ast} \in A$ for some $i$
Thus $x_{i} \wedge au \in B$ and so $x_{i} \in B$.
The proof of the final assertion is straightforward.

(4) Immediate by part (2) above.
\end{proof}

\begin{remark}{\em 
The above lemmas establish that in a restriction quantal frame all
completely prime filters are of the form $(aA)^{\ast}$ where $a$ is a partial isometry,
$a^{\ast} \in A$ and $A \cap e^{\downarrow}$ is a completely prime filter in $e^{\downarrow}$. 
}
\end{remark}

We are now ready to prove that $\mathsf{C}(Q)$ is a category.
Let $A$ and $B$ be completely prime filters in $Q$.
If $A^{\ast} = B^{+}$ define $A \cdot B = (AB)^{\uparrow}$; otherwise, the product is not defined.

\begin{lemma}\label{lem:truss} 
With the above definition, $A \cdot B$ is a completely prime filter.
\end{lemma}
\begin{proof} Suppose that $0 \in A \cdot B$.
Then $ab \leq 0$ where $a \in A$ and $b \in B$.
It follows that $ab = 0$.
Thus $(ab)^{\ast} = 0$.
That is $(a^{\ast}b)^{\ast} = 0$.
But $a^{\ast} = b_{1}^{+}$ for some $b_{1} \in B$.
It follows that $(b_{1}^{+}b)^{\ast} = 0$.
But $b_{1}^{+}b \in B$.
It follows that there is an element $b_{2} \in B$ such that $b_{2}^{\ast} = 0$.
But then $b_{2} = 0$ and $B$ is a proper filter by assumption, so that we get a contradiction.
It follows that $0 \notin A \cdot B$.
By definition, $A \cdot B$ is upwardly closed.
We prove that $A \cdot B$ is downwards directed.
Let $x,y \in A \cdot B$.
Then $a_{1}b_{1} \leq x$ and $a_{2}b_{2} \leq y$ for some $a_{1}, a_{2} \in A$ and $b_{1}, b_{2} \in B$.
Let $a \in A$ be such that $a \leq a_{1},b_{1}$ and let $b \in B$ be such that $b \leq b_{1},b_{2}$.
It follows that $ab \leq a_{1}b_{1}, a_{2}b_{2}$ and so $ab \leq x,y$.
We prove that $A \cdot B$ is completely prime.
Suppose that $\bigvee_{i \in I} x_{i} \in A \cdot B$.
Then $ab \leq \bigvee_{i \in I} x_{i}$ where, without loss of generality, we may assume
that $a$ and $b$ are partial isometries and so their product is a partial isometry.
It follows that $ab = \bigvee_{i \in I} (ab \wedge x_{i})$.
Observe that $ab \wedge x_{i} \leq ab$, where $ab$ is a partial isometry.
It follows that $ab \wedge x_{i} = ab(ab \wedge x_{i})^{\ast}$ for each $i \in I$.
We now use the fact that $\ast$ is sup-preserving, 
to deduce that
$(ab)^{\ast} = \bigvee_{i \in I} (ab \wedge x_{i})^{\ast}$. 
But $(ab)^{\ast} = (a^{\ast}b)^{\ast} = (b_{1}^{+}b)^{\ast}$ using the fact that $A^{\ast} = B^{+}$.
We have that $b_{1}^{+}b \in B$ by Lemma~\ref{lem:thatcher}.
Thus $(b_{1}^{+}b)^{\ast} \in B^{\ast}$ which is a completely prime filter in $e^{\downarrow}$ by
Lemma~\ref{lem:callghan}.
It follows that $(ab \wedge x_{i})^{\ast} \in B^{\ast}$ for some $i \in I$.
Now, $b \in B$ and $(ab \wedge x_{i})^{\ast} \in B^{\ast}$ and so by Lemma~\ref{lem:thatcher},
we have that $b(ab \wedge x_{i})^{\ast} \in B$.
Thus $ab(ab \wedge x_{i})^{\ast} \in AB$.
We have therefore proved that $ab \wedge x_{i} \in AB$
and so $x_{i} \in (AB)^{\uparrow}$.
\end{proof}

The proof of the following is strightforward. 

\begin{lemma}\label{lem:wellington} Let $Q$ be a restriction quantal frame.
Then $A \cdot \mathbf{d}(A) = A$ and $\mathbf{r}(A) \cdot A = A$.
\end{lemma}

\begin{proposition}\label{prop:one} Let $Q$ be an restriction quantal frame.
Then $(\mathsf{C}(Q),\cdot)$ is a category, the identities of which are the identity completely prime filters.
\end{proposition}
\begin{proof} Observe that $(A \cdot B)^{\ast} = B^{\ast}$ and $(A \cdot B)^{+} = A^{+}$.
It follows that $\mathbf{d}(A \cdot B) = \mathbf{d}(B)$ and $\mathbf{r}(A \cdot B) = \mathbf{r}(A)$.
From this we see that $(A \cdot B) \cdot C$ is defined if and only if $A \cdot (B \cdot C)$ is defined.
It is routine to check that we obtain the same answer in either case.
We now locate the identities of this category.
Suppose first that $F^{\uparrow}$ is a completely prime filter that contains projections by Lemma~\ref{lem:pitt}
and that $A \cdot F^{\uparrow}$ is defined.
Then $(A^{\ast})^{\uparrow} = F^{\uparrow}$.
Thus $A^{\ast} = F$.
It follows that $F^{\uparrow} = \mathbf{d}(A)$ and so $A \cdot F^{\uparrow} = A$.
It follows that all completely prime filters that contain projections are identities.
Now suppose that $B$ is an identity and that $A \cdot B$ is defined.
Then $A \cdot B = A$.
But $A \cdot \mathbf{d}(A)$ is defined and so, by associativity,
$\mathbf{d}(A) \cdot B$ is defined.
We have already proved that $\mathbf{d}(A)$ is an identity.
It follows that $\mathbf{d}(A) = B$.
\end{proof}

We shall now endow  $\mathsf{C}(Q)$ with a topology that will make it into an \'etale category.\\

\noindent
{\bf Definition. }Let $Q$ be a restriction quantal frame.
If $a$ is any element of $Q$, define $\mathscr{X}_{a}$ to be the set of all completely prime filters that contain $a$.\\

\begin{lemma}\label{lem:merkel-one} Let $Q$ be a restriction quantal frame with top element $1$ and identity $e$.
\begin{enumerate}
\item $\mathscr{X}_{1}$ is the set of all completely prime filters.
\item $\mathscr{X}_{e}$ is the set of all identity completely prime filters.
\item $\bigcup_{i \in I} \mathscr{X}_{a_{i}} = \mathscr{X}_{\bigvee_{i \in I} a_{i}}$.
\item $\mathscr{X}_{a} \cap \mathscr{X}_{b} = \mathscr{X}_{a \wedge b}$.
\end{enumerate}
\end{lemma}
\begin{proof} (1) The top element is above every element and so every completely prime filter contains the top element.

(2) The completely prime filters containing the monoid identity are prescisely the identity completely prime filters.

(3) This relies on the defining property of completely prime filters.

(4) This relies on the fact that as filters, completely prime filters are downwards directed. 
\end{proof}

Put $\tau = \{\mathscr{X}_{a} \colon a \in \mathsf{PI}(Q)\}$.
Observe that $\tau$ consists of those elements $\mathscr{X}_{a}$ where $a$ is a partial isometry.

\begin{lemma}\label{lem:base}
$\tau$ is the base for a topology on $\mathsf{C}(Q)$.
\end{lemma}
\begin{proof} Let $A$ be an arbitrary completely prime filter.
Then $A$ contains a partial isometry, $a$ say.
Thus $A \in \mathscr{X}_{a}$.
Let $\mathscr{X}_{a} \cap \mathscr{X}_{b}$ have a non-empty intersection.
Let $A \in \mathscr{X}_{a} \cap \mathscr{X}_{b}$.
Then $a,b \in A$.
It follows that $a \wedge b \in A$.
But $a \wedge b \leq a$ and $a$ is a partial isometry and so $a \wedge b$ is a partial isometry
since the partial isometries form an order-ideal.
Observe that $A \in \mathscr{X}_{a \wedge b} \subseteq \mathscr{X}_{a} \cap \mathscr{X}_{b}$.
We have therefore proved that $\tau$ is a base for a topology on $\mathsf{C}(Q)$.
\end{proof}

We endow $\mathsf{C}(Q)$ with the topology that has $\tau$ as a base;
this is the only topology we shall consider on $\mathsf{C}(Q)$.

\begin{lemma}\label{lem:open-sets} 
Let $Q$ be a restriction quantal frame.
Then $\mathscr{X}_{a}$ is an open set for any $a$.
\end{lemma}
\begin{proof} If $a$ is a partial isometry then the result follows by definition.
Suppose that $a$ is not a partial isometry.
Then $a = \bigvee_{i \in I} a_{i}$ where the $a_{i}$ are partial isometries.
Any completely prime filter containing $a$ must contain at least one of the $a_{i}$.
On the other hand, any completely prime filter containing one of the $a_{i}$
must contain $a$.
It follows that 
$$\mathscr{X}_{a} = \bigcup_{i \in I} \mathscr{X}_{a_{i}}.$$
\end{proof}

\begin{lemma}\label{lem:topology-cat} Let $Q$ be a restriction quantal frame.
Then $\mathsf{C}(Q)$ is a topological category.
\end{lemma}
\begin{proof} We prove first that the mutiplication map $\mathbf{m}$ is continuous.
Let $s$ be a partial isometry.
Then
$$\mathbf{m}^{-1}(\mathscr{X}_{s}) = \left(\bigcup_{0 \neq bc \leq s} \mathscr{X}_{b} \times \mathscr{X}_{c} \right) \cap \left(\mathsf{C}(Q) \ast \mathsf{C}(Q)\right)$$
where $b$ and $c$ are partial isometries.
It remains to prove that the maps $\mathbf{d}$ and $\mathbf{r}$ are continuous.
It is enough to prove that $\mathbf{d}$ is continuous, since the proof that $\mathbf{r}$ is continuous follows by symmetry.
We have to calculate $\mathbf{d}^{-1}(\mathscr{X}_{s})$.
Suppose that the only projection below $s$ is $0$.
Then $\mathbf{d}^{-1}(\mathscr{X}_{s}) = \varnothing$;
to see why, suppose that this set were non-empty.
Then there would exist a completely prime filter $A$ such that $s \in \mathbf{d}(A) = (A^{\ast})^{\uparrow}$.
It follows that there is a non-zero projection $f \leq s$, which is a contradiction.
Now, suppose that there are non-zero projections below $s$.
Since arbitrary-sups exist and the sup of a set of projections is a projection,
there is a largest projection $f$ such that $f \leq s$.
Observe that $\mathscr{X}_{f} \subseteq \mathscr{X}_{s}$ 
and that any completely prime filter that contains $s$ and a projection must belong to $\mathscr{X}_{f}$.
It follows that $\mathbf{d}^{-1}(\mathscr{X}_{s}) = \mathbf{d}^{-1}(\mathscr{X}_{f})$.
It remains to calculate $\mathbf{d}^{-1}(\mathscr{X}_{f})$ where $f$ is any projection.
Let $A$ be a completely prime filter such that $\mathbf{d}(A) \in \mathscr{X}_{f}$.
Let $a \in A$ be any partial isometry.
Then $A = (afF^{\uparrow})$.
It follows that the set $\mathbf{d}^{-1}(\mathscr{X}_{f})$ 
is the union of all sets of the form $\mathscr{X}_{af}$ where $a$ is a partial isometry and $af \neq 0$.
\end{proof}

\begin{lemma}\label{lem:topology-open} 
The maps $\mathbf{d}$ and $\mathbf{r}$ are open.
\end{lemma}
\begin{proof} We prove that the map $\mathbf{d}$ is open;
the proof that $\mathbf{r}$ is open follows by symmetry.
Observe that $\mathbf{d}(\mathscr{X}_{s}) = \mathscr{X}_{s^{\ast}}$.
\end{proof}

The proof of the following is now immediate
once you observe that the sets $\mathscr{X}_{s}$ are open local bisections when $s$ is a partial isometry.

\begin{proposition}\label{prop:two} If $Q$ is a restriction quantal frame then
$\mathsf{C}(Q)$ is an \'etale topological category.
\end{proposition}

We can describe the identity space of $\mathsf{C}(Q)$.

\begin{lemma}\label{lem:rishi} Let $Q$ be a restriction quantal frame with identity $e$.
Then the identity space of $\mathsf{C}(Q)$ is homeomorphic to the space $\mathsf{pt}(e^{\downarrow})$.
\end{lemma}
\begin{proof} We use the order isomorphism established in Proposition~\ref{prop:bijection-cp}.
We now prove that this bijection is a homeomorphism.
Let $\mathscr{X}_{a}$ contain an identity completely prime filter.
Then $a$ is above a projection.
It follows that we may restrict out attention to open sets of the form $\mathscr{X}_{f}$
where $f$ is a projection.
Observe that $A \in \mathscr{X}_{f}$ if and only if $A \cap e^{\downarrow} \in X_{f}$.
This establishes our homeomeorphism.
 \end{proof}

\begin{lemma}\label{lem:restriction-functor} Let $\phi \colon R \rightarrow S$ 
be a morphism of restriction quantal frames.
Then this induces a continuous covering functor $\phi^{-1} \colon \mathsf{C}(S) \rightarrow \mathsf{C}(R)$.
\end{lemma}
\begin{proof} If $B$ is a completely prime filter in $S$, 
then $\phi^{-1}(B)$ is a completely prime filter in $R$ and so $\phi^{-1}$ induces a continous function by \cite[Theorem 1.4]{J}.

It remains to show that $\phi^{-1}$ is a covering functor.
This will be proved in stages.

Observe that $\phi^{-1}$ maps identity completely prime filters to identity completely prime filters;
we shall prove this.
Suppose that $B$ contains a projection $f$.
We need to prove that $\phi^{-1}(B)$ contains a projection.
Let $x \in \phi^{-1}(B)$.
Then $\phi (x) \in B$.
Thus, since $B$ contains a projection, it follows by Lemma~\ref{lem:thanos},
that $\phi (x)^{\ast} \in B$.
But, $\theta$ is a morphism of Ehresmann semigroups and so $\phi (x)^{\ast} = \phi (x^{\ast})$.
We deduce that $x^{\ast} \in \phi^{-1}(B)$. 
Thus $\theta^{-1}(B)$ contains a projection.

Let $B$ be a completely prime filter in $S$.
By the above, $A = \phi^{-1}(B)$ is a completely prime filter in $R$.
We now describe the completely prime filters $A$ and $B$.
We may write $A = (aA^{\ast})^{\uparrow}$ where $a \in A$ and $a$ is a partial isometry.
But $a \in A = \phi^{-1}(B)$ and so $b = \phi (a) \in B$,
and $\phi$ preserves partial isometries.
Thus $b$ is also a partial isometry.
It follows that $B = (bB^{\ast})^{\uparrow}$.
We have used Lemma~\ref{lem:thatcher} twice.

We now prove that 
$$\mathbf{d}(\phi^{-1}(B)) = \phi^{-1}(\mathbf{d}(B)),$$
with the proof of the dual result following by symmetry.
We therefore actually need to prove that 
$$(\phi^{-1}(B)^{\ast})^{\uparrow} = \phi^{-1}((B^{\ast})^{\uparrow}).$$
We prove first that 
$$(\phi^{-1}(B)^{\ast})^{\uparrow} \subseteq \phi^{-1}((B^{\ast})^{\uparrow}).$$
Let $r \in (\phi^{-1}(B)^{\ast})^{\uparrow}$.
Then $u^{\ast} \leq r$ where $u \in A = \phi^{-1}(B)$.
Apply $\phi$ to both sides to obtain
$\phi (u^{\ast}) = \phi (u)^{\ast} \leq \phi (r)$
where we have used the fact that $\phi$ is a morphism of Ehresmann semigroups.
But $\phi (u) \in B$.
Using our description of elements of $B$ above,
it follows that $\phi (u) \geq bf$ where $f \in B^{\ast}$.
Thus $(bf)^{\ast} \leq \phi (r)$ and so $r \in \phi^{-1}((B^{\ast})^{\uparrow})$.
We now prove the reverse inclusion.
That is: we prove that 
$$\phi^{-1}(\mathbf{d}(B)) \subseteq \mathbf{d}(\phi^{-1}(B)).$$
Let $x \in \phi^{-1}(\mathbf{d}(B))$.
Observe that $\mathbf{d}(B)$ is a completely prime filter that contains projections.
Thus $\phi^{-1}(\mathbf{d}(B))$ is a completely prime filter that contains projections
by what we proved above.
There is therefore a projection $f \in \phi^{-1}(\mathbf{d}(B))$ such that $f \leq x$.
By definition, $\phi (f) \in \mathbf{d}(B)$.
Observe that $\phi (af) = \phi(a) \phi(f) \in B\mathbf{d}(B) \subseteq B$.
It follows that $af \in \phi^{-1}(B)$.
Now, $(af)^{\ast} = a^{\ast}f \leq f$.
We have therefore shown that $(af)^{\ast} \leq x$ and $\phi (af) \in B$.
Thus $x \in \mathbf{d}((\phi^{-1}(B))$.

We can now show that $\theta^{-1}$ is functor.
Let $B_{1},B_{2}$ be completely prime filters in $S$ such that
$\mathbf{d}(B_{1}) = \mathbf{r}(B_{2})$.
Then $B_{1} \cdot B_{2} = (B_{1}B_{2})^{\uparrow}$ is a completely prime filter in $S$.
It follows that $\phi^{-1}(B_{1} \cdot B_{2})$ is a completely prime filter in $R$.
By the above, 
$\mathbf{d}(\phi^{-1}(B_{1})) = \phi^{-1}(\mathbf{d}(B_{1}))$ 
and
$\mathbf{r}(\phi^{-1}(B_{2})) = \phi^{-1}(\mathbf{r}(B_{2}))$.
It follows that $\mathbf{d}(\phi^{-1}(B_{1})) = \mathbf{r}(\phi^{-1}(B_{2}))$.
Thus $\phi^{-1}(B_{1}) \cdot \phi^{-1}(B_{2})$ is a well-defined completely prime filter in $R$.
It remains to prove that 
$\phi^{-1}(B_{1} \cdot B_{2})
=
\phi^{-1}(B_{1}) \cdot \phi^{-1}(B_{2})$.
It is easy to check that
$\phi^{-1}(B_{1}) \cdot \phi^{-1}(B_{2})
\subseteq 
\phi^{-1}(B_{1} \cdot B_{2})$.
But both sides are completely prime filters and $\mathbf{d}$ of each side is the same.
It follows by Lemma~\ref{lem:thatcher} that the two completely prime filters are equal.

It only remains to check that $\phi^{-1}$ is a covering (functor).
Suppose that $A$ and $B$ are completely prime filters in $S$ such that $\phi^{-1}(A) = \phi^{-1}(B)$
and $\mathbf{d}(A) = \mathbf{d}(B)$.
We prove that $A = B$.
Let $r \in \phi^{-1}(A) = \phi^{-1}B)$ be any partial isometry.
Then $\phi (r) \in A$ and $\phi(r) \in B$.
It follows that $A$ and $B$ have the same domains and share a partial isometry.
Thus, by Lemma~\ref{lem:thatcher}, we have that $A = B$.
Let $F$ be an identity completely prime filter in $S$.
Then $\phi^{-1}(F)$ is an identity prime filter in $R$.
Suppose that $A$ is a completely prime filter in $R$ such that $\mathbf{d}(A) = \phi^{-1}(F)$.
Let $a$ be any partial isometry in $R$ such that $a \in A$.
Then $a^{\ast} \in  \phi^{-1}(F)$ and $\phi (a)$ is a partial isometry in $S$.
Observe that $\phi(a^{\ast}) \in F$.
We may therefore construct the completely prime filter $B = (\phi (a)F)^{\uparrow}$, by Lemma~\ref{lem:thatcher}.
We claim that $A = \phi^{-1}(B)$; 
this follows from the fact that $\phi (a) \in B$ and so $a \in \phi^{-1}(B)$
and $\mathbf{d}(\phi^{-1}(B)) = \phi^{-1}(\mathbf{d}(B)) = \phi^{-1}(F)$.
What we have done above for $\mathbf{d}$, we can, by symmetry, do for $\mathbf{r}$.
It follows that $\phi^{-1}$ is a covering (functor).
\end{proof}

Extend $\mathsf{C}$ to morphisms.
By Proposition~\ref{prop:two} and Lemma~\ref{lem:restriction-functor},
we have proved the following.

\begin{proposition}\label{prop:functor-two} $\mathsf{C}$ is a contravariant functor 
from the category whose objects are the restriction quantal frames
and whose arrows are the morphisms of restriction quantal frames to the category whose objects are the \'etale topological categories
and whose arrows are the continuous covering functors
\end{proposition}

\section{There and back again}

In Proposition~\ref{prop:kemi}, we proved that if $C$ was an \'etale category then $\Omega (C)$ was
a restriction quantal frame.
In Proposition~\ref{prop:two}, we proved that if $Q$ was a restriction quantal frame
then $\mathsf{C}(Q)$ was an \'etale category.
We now iterate these two constructions.
Let $Q$ be a restriction quantal frame.
Then $\mathsf{C}(Q)$ is an \'etale category.
It follows that $\Omega (\mathsf{C}(Q))$ is an Ehresmann quantal frame.
Define a map $\chi \colon Q \rightarrow \Omega (\mathsf{C}(Q))$
by $\chi (a) = \mathscr{X}_{a}$, which is an open set by Lemma~\ref{lem:open-sets}. 

\begin{lemma}\label{lem:merkel-two} Let $Q$ be a restriction quantal frame with top element $1$ and identity $e$.
\begin{enumerate}
\item $(\mathscr{X}_{a})^{\ast} = \mathscr{X}_{a^{\ast}}$, and dually, in the restriction quantal frame $\Omega (\mathsf{C}(Q))$.
\item If $a$ is a partial isometry then $\mathscr{X}_{a}$ is a partial isometry in the restriction quantal frame $\Omega (\mathsf{C}(Q))$.
\end{enumerate}
\end{lemma}
\begin{proof} (1) By definition, $(\mathscr{X}_{a})^{\ast} = \{ \mathbf{d}(A) \colon A \in \mathscr{X}_{a}\}$.
This is precisely the set of all completely prime filters containing $a^{\ast}$ by Lemma~\ref{lem:thatcher}.
The dual result follows similarly.

(2) Let $A,B \in \mathscr{X}_{a}$ where $\mathbf{d}(A) = \mathbf{d}(B)$.
Then by Lemma~\ref{lem:thatcher}, we must have that $A = B$.
The dual result is proved similarly.
\end{proof}

\begin{lemma}\label{lem:morphism-hobbit}
The function $\chi$ is a morphism of restriction quantal frames.
\end{lemma}
\begin{proof} To prove that $\chi$ is a semigroup homomorphism,
we need to prove that $\mathscr{X}_{ab} = \mathscr{X}_{a}\mathscr{X}_{b}$.
We prove first that 
$\mathscr{X}_{a}\mathscr{X}_{b} \subseteq \mathscr{X}_{ab}$.
Let $A \in  \mathscr{X}_{a}$ and $B \in \mathscr{X}_{b}$ and suppose that $A \cdot B$ is defined.
Then $A \cdot B = (AB)^{\uparrow}$ and so $A \cdot B$ contains the element $ab$.
We prove the reverse inclusion $\mathscr{X}_{ab} \subseteq \mathscr{X}_{a}\mathscr{X}_{b}$.
Let $A \in \mathscr{X}_{ab}$.
We may write $a$ as a product of partial isometries $a_{i}$ so that $a = \bigvee_{i \in I} a_{i}$;
similarly, we may write $b =  \bigvee_{j \in J} b_{j}$ of partial isometries.
It follows that $ab = \bigvee a_{i}b_{j}$, since we are working in a quantale.
It follows that $a_{i}b_{j} \in A$ for some $i$ and some $j$.
Observe that since $Q$ is a restriction quantal frame,
it follows that $a_{i}b_{j}$ is a partial isometry and that $a_{i} \leq a$ and $b_{j} \leq b$.
Thus it is enough to prove that 
$\mathscr{X}_{ab} \subseteq \mathscr{X}_{a}\mathscr{X}_{b}$
under the assumption that $a$ and $b$ are partial isometries.
Let $C \in \mathscr{X}_{ab}$. 
Put $B = (b(C^{\ast})^{\uparrow})$ where $b^{\ast} \in C^{\ast}$.
Then $B$ is a completely prime filter by Lemma~\ref{lem:thatcher}.
Put $A = (a(B^{+})^{\uparrow})^{\uparrow}$.
Observe that by construction $\mathbf{d}(A) = \mathbf{r}(B)$
and that $C = A \cdot B$.
We have therefore proved the result.
It remains to show that $\chi$ satisfy the four conditions that define a morphism of restriction quantal frames.
This follows by Lemma~\ref{lem:merkel-one} and Lemma~\ref{lem:merkel-two}.
\end{proof}

We say that a restriction quantal frame $Q$ is {\em spatial} if $\chi$ is an isomorphism.

\begin{lemma}\label{lem:spatial-eqf}
Let $Q$ be an Ehresmann quantal frame.
Then $Q$ is spatial if and only if $\mathscr{X}_{a} = \mathscr{X}_{b}$ implies that $a = b$.
\end{lemma}
\begin{proof} The proof of one direction is trivial.
It is clear that if the condition holds then $\chi$ is injective.
We now prove that it is surjective.
Let $U$ be an element of $\Omega (\mathsf{C}(Q))$.
Thus $U$ is an open set in $\mathsf{C}(Q)$.
It follows that $U$ is a union of sets of the form $\mathscr{X}_{q}$ where $q \in Q$ and $q$ is a partial isometry.
Thus $U$ is itself a set of the form $\mathscr{X}_{q}$ for some $q \in Q$.
\end{proof}

The following was first proved as \cite[Lemma 8.3]{KL}.

\begin{proposition}\label{prop:spatial-eqf} Let $Q$ be an Ehresmann quantal frame.
Then $Q$ is spatial if and only if $e^{\downarrow}$ is spatial.
\end{proposition}
\begin{proof} Suppose first that $Q$ is spatial.
We prove that $e^{\downarrow}$ is spatial.
Let $f,g$ be any projections such that $f \nleq g$.
We prove that there is a completely prime filter in $e^{\downarrow}$ that contains $f$ but omits $g$.
There is a completely prime filter $A$ in $Q$ that contains $f$ but omits $g$.
But $A$ is a completely prime filter that contains a projection.
It follows that $A \cap e^{\downarrow}$ is a completely prime filter in $e^{\downarrow}$ that contains $f$ but omits $g$.
This proves that $e^{\downarrow}$ is spatial.

Now, assume that $e^{\downarrow}$ is spatial.
We prove that $Q$ is spatial.
By Lemma~\ref{lem:spatial-eqf}, it is enough to prove that if $\mathscr{X}_{a} = \mathscr{X}_{b}$ then $a = b$.
We begin by assuming that $a$ and $b$ are partial isometries.
Suppose that $\mathscr{X}_{a} = \mathscr{X}_{b}$. 
Then $\mathscr{X}_{a} = \mathscr{X}_{a \wedge b}$. 
It follows that $\mathscr{X}_{a^{\ast}} = \mathscr{X}_{(a \wedge b)^{\ast}}$.
We now use the fact that $e^{\downarrow}$ is spatial and the fact that
identity completely prime filters are determined by the (completely prime) set of projections they contain
to deduce that $a^{\ast} = (a \wedge b)^{\ast}$.
Now $a \wedge b \leq a$ and $a$ is a partial isometry.
It follows that $a \wedge b = a(a \wedge b)^{\ast} = a$.
We have proved that $a \leq b$.
By symmetry, we deduce that $a = b$.
Now we can deal with the general case. 
Suppose that $\mathscr{X}_{a} = \mathscr{X}_{b}$. 
Then $a = \bigvee_{i \in I} a_{i}$ where the $a_{i}$ are partial isometries.
We have that $\mathscr{X}_{a_{i}} \subseteq \mathscr{X}_{b}$ so that
$\mathscr{X}_{a_{i}} = \mathscr{X}_{a_{i} \wedge b}$.
Now $a_{i}$ is a partial isometry and $a_{i} \wedge b$ is a partial isometry,
since the partial isometries form an order-ideal.
We deduce from the above that $a_{i} = a_{i} \wedge b$.
We have proved that $a_{i} \leq b$ for any $i$.
We have therefore proved that $a \leq b$.
By symmetry, we deduce that $a = b$. 
\end{proof}

Let $C$ be an \'etale topological category.
Then $\Omega (C)$ is a restriction quantal frame.
It follows that $\mathsf{C}(\Omega (C))$ is an \'etale topological category.
Define a map $\omega \colon C \rightarrow \mathsf{C}(\Omega (C))$
by $\omega (x) = O_{x}$, all the open sets of $C$ that contain $x$.
This is a completely prime filter.
We prove the following.

\begin{lemma}\label{lem:morphism-cont-cov-functor}
The function $\omega$ is a continuous covering functor.
\end{lemma}
\begin{proof} We prove first that $\omega$ is a functor.
Suppose that $x$ is an identity of the category $C$.
Then $O_{x}$ is a completely prime filter that contains $C_{o}$.
Thus $O_{x}$ is an identity completely prime filter.
Let $U$ be an open set containing $x$.
Then $\mathbf{d}(U)$ is an open subset of $C_{o}$ containing $\mathbf{d}(x)$.
Let $V$ be any open set of $C$ containing $\mathbf{d}(x)$.
Then $V \cap C_{o}$ is an open subset of $C_{o}$ containing $\mathbf{d}(x)$.
We have therefore proved that $\mathbf{d}(O_{x}) = O_{\mathbf{d}(x)}$, and dually.
Suppose that $\mathbf{d}(x) = \mathbf{r}(y)$ in the category $C$.
Then by the above, we have that $\mathbf{d}(O_{x}) = \mathbf{r}(O_{y})$.
Clearly, $O_{x}O_{y} \subseteq O_{xy}$ since the product of open sets is an open set.
It follows that $O_{x} \cdot O_{y} \subseteq O_{xy}$.
But $O_{x} \cdot O_{y}$ and $O_{xy}$ are completely prime filters,
one contained in the other,
having the same value when $\mathbf{d}$ is applied.
It follows that $O_{x} \cdot O_{y} = O_{xy}$.
This proves that $\omega$ is a functor.
A base for the topology on $\mathsf{C}(\Omega (C))$ is given by sets of the form $\mathscr{X}_{U}$
where $U$ is an open subset of $C$.
Observe that $\omega^{-1}(\mathscr{X}_{U}) = U$.
It follows that $\omega$ is continuous.
Let $x,y \in C$ such that $\mathbf{d}(x) = \mathbf{d}(y)$ and $O_{x} = O_{y}$.
We prove that $x = y$; a dual result holds for the case wehere we assume instead that $\mathbf{r}(x) = \mathbf{r}(y)$
A base for the topology on $C$ is the set of open local bisections.
It follows that there is an open local bisection $U$ containing both $x$ and $y$.
we deduce that $x = y$.
We have proved that $\omega$ is $\mathbf{d}$-injective.
Now, suppose that $f$ is an identity of $C$ such that $\mathbf{d}(F^{\uparrow}) = \omega (f) = O_{f}$
and $\mathbf{d}(A) = F^{\uparrow}$.
Let $U$ be an open local bisection that belongs to $A$.
Then there is a unique element $x \in U$ such that $\mathbf{d}(x) = e$.
Then completely prime filter $O_{x}$ contains $U$ and $\mathbf{d}(O_{x}) = O_{f}$.
It follows that $\omega (x) = A$.
Similar results can be proved for $\mathbf{r}$.
We have therefore proved that $\omega$ is a covering (functor).\end{proof}

An \'etale topological category is said to be {\em sober} if $\omega$ is 
an isomorphism of catgeories and a homeomorphism of topological spaces.

\begin{lemma}\label{lem:sober-ec} Let $C$ be an \'etale topological category.
Then $C$ is sober if and only if the map $\omega$ is a bijection.
\end{lemma}
\begin{proof} Only one direction needs proving.
We suppose that the function $\omega$ is a bijection.
It is already a functor and so it is an isomorphism of categories.
Thus the result will be proved if we can show that $\omega$ is open.
Let $U$ be an open set of $C$.
By definition $\omega (U) = \{O_{c} \colon c \in U\}$,
where each $O_{c}$ is the set of all open sets that contain $c$.
For each $c$, we have that $U \in O_{c}$.
Thus, each element of $\omega (U)$ is a completely prime filter containing $U$.
We may therefore consider the set of {\em all} completely prime filters in $\Omega (C)$ that contain $U$
which is $\mathscr{X}_{U}$.
It follows that $\omega (U) \subseteq \mathscr{X}_{U}$.
On the other hand, let $A \in \mathscr{X}_{U}$.
Thus $A$ is a completely prime filter that contains $U$.
By assumption, $\omega$ is a bijection.
Thus $A = \omega (x)$ for some $x$.
Thus $A = O_{x}$ for some $x \in C$.
But $x$ belongs to every open set in $A$ and so belongs to $U$.
It follows that $A \in \omega (U)$.
We have therefore proved that $\omega (U) = \mathscr{X}_{U}$.
It follows that $\omega$ is open.
\end{proof}

The following was first proved as \cite[Lemma 8.2]{KL}.

\begin{proposition}\label{prop:sober-ec} Let $C$ be an \'etale category.
Then $C$ is sober if and only if $C_{o}$ is sober.
\end{proposition}
\begin{proof} The space $C_{o}$ is sober if the function $C_{o} \rightarrow \Omega (C_{o})$ is a bijection.
Suppose first that $C$ is sober.
We prove that $C_{o}$ is sober.
Denote the set of open sets in $C_{o}$ that contain $e \in C_{o}$ by $\widetilde{O_{e}}$.
Let $e,f \in C_{o}$ and suppose that $\widetilde{O_{e}} = \widetilde{O_{f}}$.
Then $O_{e} = O_{f}$; to see why, let $V$ be any open set of containing $e$.
Then $V \cap C_{o}$ is an open set of $C_{o}$ that contains $e$.
Thus $V \cap C_{o}$ is an open set of $C_{o}$ that contains $f$.
It follows that $V$ is an open set of $C$ that contains $f$.
This proves one inclusion and the reverse inclusion follows by symmetry.
By definition, $\omega (e) = \omega (f)$.
Thus $e = f$, as required.
Suppose now that $F$ is any completely prime filter of open sets of $C_{o}$.
Then $F^{\uparrow}$ is a completely prime filter of open sets of $C$.
By assumption, $F^{\uparrow} = \omega (x)$ for some $x \in C$.
But $F^{\uparrow}$ contains $F$ which consists of open subsets of the set of identities.
This means that $x = e$ is an identity.
It follows that $F = \widetilde{O_{e}}$.
We have therefore proved that $C_{o}$ is sober if $C$ is.

We now assume that $C_{o}$ is sober and prove that $C$ is sober.
We use our description of completely prime filters in $C$.
These have the form $(aF)^{-1}$ where $F$ is a completely prime filter in $e^{\downarrow}$,
$a$ is a partial isometry and $a^{\ast} \in F$.
By assumption, $F$ is of the form $\widetilde{O_{e}}$ where $e \in C_{o}$.
The partial isometries in \'etale categories are open local bisections.
It follows that there is a unique element $x \in a$ such that $\mathbf{d}(x) = e$.
It follows that $O_{e} \subseteq (aF)^{-1}$ but these two completely prime filters are equal since $(O_{e})^{\ast} = F$. 
Sobreity now follows.
\end{proof}

\begin{lemma}\label{lem:heath}\mbox{}
\begin{enumerate}
\item For each \'etale category $C$, the restriction quantal frame $\Omega (C)$ is spatial.
\item For each restriction quantal frame $Q$, the \'etale category $\mathsf{C}(Q)$ is sober. 
\end{enumerate}
\end{lemma}
\begin{proof} 
(1) This follows from the fact  $\Omega (C)$ is spatial if and only if $\Omega (C_{o})$ is spatial by Proposition~\ref{prop:spatial-eqf},
since $\Omega (C_{o})$ is precisely the set of projections of $\Omega (C)$.
But now we use Lemma~\ref{lem:sobriety-spatiality}, which enables us to deduce that $\Omega (C_{o})$ is always spatial. 

(2) This follows from the fact that $\mathsf{C}(Q)$ is sober if and only if  $\mathsf{C}(Q)_{o}$ is sober by
Proposition~\ref{prop:sober-ec}.
However, by Lemma~\ref{lem:rishi}, we have that $\mathsf{C}(Q)_{o}$ is homeomorphic with $\mathsf{pt}(e^{\downarrow})$.
But now we use Lemma~\ref{lem:sobriety-spatiality}, which enables us to deduce that  $\mathsf{pt}(e^{\downarrow})$ is always sober.\end{proof}

Denote by \mbox{\bf RQF} the category of restriction quantal frames and their morphisms.
Denote by \mbox{\bf EC} the category of \'etale categories and their morphisms.

\begin{theorem}[Adjunction Theorem I]\label{them:adjunction} The functor 
$$\mathsf{C} \colon \mbox{\bf RQF}^{op} \rightarrow \mbox{\bf EC}$$
is right adjoint to the functor 
$$\Omega \colon \mbox{\bf EC} \rightarrow \mbox{\bf RQF}^{op}.$$
\end{theorem}
\begin{proof} Given a continuous covering functor $\alpha \colon C \rightarrow \mathsf{C}(Q)$
we may construct the morphism $\alpha^{-1} \chi \colon Q \rightarrow \Omega (C)$.
Likewise, given a morphism $\beta \colon Q \rightarrow \Omega (C)$,
we may construct a morphism $\beta^{-1} \omega \colon C \rightarrow \mathsf{C}(Q)$.
We shall show that these two constructions are the inverses of each other.

Let $\alpha \colon C \rightarrow \mathsf{C}(Q)$ be a continuous covering functor.
We calculate what $(\alpha^{-1}\chi)^{-1}\omega$ does to elements of $C$.
Let $c \in C$.
Then $q \in ((\alpha^{-1} \chi)^{-1} \omega)(c)$ if and only if $q \in (\alpha^{-1} \chi)^{-1}(O_{c})$
if and only if $(\alpha^{-1}\chi)(q) \in O_{c}$
if and only if $c \in (\alpha^{-1}\chi)(q)$
if and only if $c \in \alpha^{-1} (\mathscr{X}_{q})$
if and only if $\alpha (c) \in \mathscr{X}_{q}$
if and only if $q \in \alpha (c)$.
It follows that $\alpha (c) =  ((\alpha^{-1} \chi)^{-1} \omega)(c)$, which is what was required.

Let $\beta \colon Q \rightarrow \Omega (C)$ be a morphism.
We calculate what $(\beta^{-1} \omega)^{-1} \chi$ does to an element of $Q$.
Let $q \in Q$.
Then $c \in ((\beta^{-1} \omega)^{-1} \chi)(q)$
if and only if $c \in (\beta^{-1} \omega)^{-1}(\mathscr{X}_{q})$
if and only if $(\beta^{-1}\omega)(c) \in \mathscr{X}_{q}$
if and only if $q \in (\beta^{-1}\omega)(c)$
if and only if $q \in \beta^{-1}(O_{c})$
if and only if $\beta(q) \in O_{c}$
if and only if $c \in \beta (q)$.
It follows that $\beta (q) = ((\beta^{-1} \omega)^{-1} \chi)(q)$, which is what was required.

Naturality is straightforward to prove here and so we have an adjunction.
\end{proof}

The following is immediate by Theorem~\ref{them:adjunction} and Lemma~\ref{lem:heath}.

\begin{corollary} The category of spatial restriction quantal frames is dually equivalent to the category of
sober \'etale categories.
\end{corollary}

\section{Complete restriction monoids and restriction quantal frames}

The material in this section is taken from Section~2.4 and Section~2.5 of \cite{KL}.
Let $\theta \colon Q \rightarrow R$ be a morphism between restriction quantal frames.
According to Proposition~\ref{prop:pi-restriction-semigroup},
both $\mathsf{PI}(Q)$ and $\mathsf{PI}(R)$ are complete restriction monoids.
The morphism $\theta$ defines a function 
$\theta' = (\theta \mid \mathsf{PI}(Q)) \colon \mathsf{PI}(Q) \rightarrow \mathsf{PI}(R)$
by restriction.
It is a monoid homomorphism; 
it is a homomorphism of Ehresmann semigroups; 
it preserves all compatible joins;
and it preserves all finite meets.
It has an extra property due to the fact that $\theta$ maps top elements to top elements.
Let the top element of $Q$ be $1_{Q}$.
By assumption, we may write $1_{Q} = \bigvee_{i \in I} q_{i}$ where the $q_{i}$ are partial isometries.
It follows that $\theta (1_{Q}) = \bigvee_{i \in I} \theta'(q_{i}) = 1_{R}$.
Let $r$ be any partial isometry of $R$.
Then $r \leq 1_{R}$.
Thus $r = 1_{R} \wedge r$.
It follows that $r = \bigvee_{i \in I} r \wedge \theta'(q_{i})$.
But the partial isometries form an order-ideal.
We have therefore proved that in $\mathsf{PI}(R)$ we may write
$r =  \bigvee_{i \in I} r_{i}$ where each $r_{i} \leq  \theta'(q_{i})$;
we describe this latter property by saying that $\theta'$ is {\em proper}.

Formally, a morphism $\theta \colon S \rightarrow T$ between complete restriction monoids is {\em proper}
if $\mbox{im} (\theta)^{\bigvee} = T$.

A {\em morphism} between complete restriction monoids is defined
to be a monoid homomorphism of Ehresmann monoids that preserves compatible joins.
A {\em completely prime filter} $A$ in a complete restriction monoid $S$ is a proper (that is, one that does not contain the bottom element) filter that is downwards
directed with the property that if $\bigvee_{i \in I}a_{i}$ is defined and  $\bigvee_{i \in I}a_{i} \in A$
then $a_{i} \in A$ for some $i \in I$.

\begin{lemma} Let $\phi \colon S \rightarrow T$ be a morphism between complete restriction monoids.
If $\phi$ is proper then every completely prime filter of $T$ contains an element of the image of $\phi$.
\end{lemma}
\begin{proof} Let $B$ be any completely prime filter in $T$.
Let $b \in B$.
By assumption, we may write $b =  \bigvee_{i \in I} b_{i}$ where each $b_{i} \leq  \phi (s_{i})$.
Since $B$ is a completely prime filter, $b_{i} \in B$ for some $i \in I$.
It follows that $\phi (s_{i}) \in B$.
\end{proof}

Recall from  \cite[Section 2.4]{KL}, that if $S$ is a complete restriction monoid then $\mathcal{L}^{\bigvee}(S)$
is the restriction quantal frame of all order-ideals of $S$ closed under arbitrary joins.
From \cite[Proposition 2.40]{KL}, it follows that if $R$ is a restriction quantal frame such that $\mathsf{PI}(R) = S$
then $\mathcal{L}^{\bigvee}(S)$ is isomorphic to $R$.

A proper morphism of complete restriction monoids that preserves finite meets will be called {\em callitic}.
Let $\theta \colon S_{1} \rightarrow S_{2}$ be a callitic morphism between complete restriction monoids.
Let $R_{1}$ be such that $\mathsf{PI}(R_{1}) = S_{1}$ and let $R_{2}$ be such that $\mathsf{PI}(R_{2}) = S_{2}$.
We may define $\Theta \colon R_{1} \rightarrow R_{2}$ by
$\Theta (a) = \bigvee_{i \in I} \theta (a_{i})$ where $a \in R_{1}$ and $a_{i} \in S_{1}$.
We can prove this is well-defined using \cite[Section 2.4]{KL}, but this can also be proved directly using 
properties of $\theta$.
The fact that $\Theta$ is a morphism of restriction quantal frames is now easy to check;
we show only that the top element of $R_{1}$ is mapped to the top element of $R_2$.
Let $1_{1}$ be the top element of $R_{1}$ and let $1_{2}$ be the top element of $R_{2}$.
We know that $1_{1}$ must be the join of all elements of $S_{1}$.
Thus $1_{1} = \bigvee_{s \in S_{1}} s$.
By definition, $\Theta (1_{1}) =  \bigvee_{s \in S_{1}} \theta (s)$.
By the same token, $1_{2} = \bigvee_{s' \in S_{2}} s'$.
However, $\theta$ is callitic.
It follows that for each $s' \in S_{2}$, we have that $s' \in \mbox{im} (\theta)^{\bigvee}$.
Whence, $1_{2} \leq \Theta (1_{1})$.
But  $\Theta (1_{1}) \leq 1_{2}$, and so $\Theta (1_{1}) = 1_{2}$.
We denote the category of complete restriction monoids and callitic morphisms by \mbox{\bf CRM}.
The following was proved as \cite[Theorem 2.44]{KL}, but also follows from the results of this section. 

\begin{theorem} 
The category \mbox{\bf CRM} is equivalent to the category \mbox{\bf RQF}. 
\end{theorem}

\section{The second adjunction theorem}

We shall use the results of the previous section, to prove the results of this one.
In particular, we shall restate Theorem~\ref{them:adjunction} and replace restriction quantal frame by complete restriction monoids.

Let $S$ be a complete restriction monoid.
A nonempty subset $A \subseteq S$ is said to be a {\em filter}
if it is closed under binary meets and closed upwards.
It is {\em proper} if $0 \notin A$.
It is {\em completely prime} if whenever $\bigvee_{i \in I} a_{i}$ exists in $S$
and $\bigvee_{i \in I} a_{i} \in A$ then $a_{i} \in A$ for some $i \in I$.  

Let $S$ be a complete restriction monoid.
Then $S$ can be regarded as a subset of a restriction quantal frame $R$ with the property that 
$\mathsf{PI}(R) = S$;
thus the set of partial isometries of $R$ is precisely the set $S$.
This is proved in \cite[Section 2.4]{KL}.
Since $S \subseteq R$ in what follows, we need some notation so that we can be clear
about which monoid we are worling in.
If $X \subseteq R$, define 
$$X\downarrow \,= X \cap \mathsf{PI}(R);$$
in other words, those elements of $S$ which belong to $X$.
If $Y \subseteq S$, define
$$Y\uparrow \,= (Y^{\uparrow})\downarrow;$$
in other words, those elements of $S$ above an element of $Y$.
Denote by $\mathsf{C}'(S)$ the set of all completely prime filters of $S$.
We denote the completely prime filters in $S$ with primes,
and the completely prime filters in $R$ without primes.

\begin{lemma}\label{lem:bijection-one} 
Let $S \subseteq R$ as above.
Then there is a bijection from 
$\mathsf{C}'(S)$ 
to
$\mathsf{C}(R)$
given by $A' \mapsto (A')^{\uparrow}$.
\end{lemma} 
\begin{proof}
Let $A'$ be a completely prime filter of $S$.
We prove that $A = (A')^{\uparrow}$ is a completely prime filter of $R$.
It is straightforward to check that $A$ is a proper filter of $R$.
Suppose that $\bigvee_{i \in I} a_{i} \in A$.
Then there is an element $a \in A'$ such that
$a \leq \bigvee_{i \in I} a_{i}$.
It follows that $a = \bigvee_{i \in I} a_{i} \wedge a$.
By assumption, $a$ is a partial isometry.
It follows that each element $a_{i} \wedge a$ is a partial isometry.
Thus $a_{i} \wedge a \in A'$ for some $i \in I$.
It follows that $a_{i} \in A$, as required.
We have therefore defined a map from  
$\mathsf{C}'(S)$ 
to
$\mathsf{C}(R)$.
Observe that $A\downarrow \, = A'$.
It follows that this map is injective.
If $A \in \mathsf{C}(R)$ then $A' = A\downarrow$ is a non-empty set,
because each element of $R$ is a join of partial isometries,
and it is easy to check that $A'$ is a completely prime filter in $S$.
Observe that $A = (A')^{\uparrow}$ in $R$.
We have therefore defined a bijection.
\end{proof}

If $a \in S$, define $\mathscr{X}_{a}'$ to be the set of completely prime filters of $S$ that contain the element $a$.
The proof of the following is straightforward.

\begin{lemma}\label{lem:bijection-two} Under the above bijection, the set $\mathscr{X}_{a}'$ 
is mapped to the set $\mathscr{X}_{a}$.
\end{lemma} 

Let $A'$ be a completely prime filter in $S$.
Define $(A')^{\ast} = \{a^{\ast} \colon a \in A'\}$ and  $(A')^{+} = \{a^{+} \colon a \in A'\}$.
Define $d(A') = (A')^{\ast}\uparrow$ and  $r(A') = (A')^{+}\uparrow$.
If $A',B'$ are both completely prime filters in $S$,
define $A' \bullet B' = (A'B')\uparrow$ if and only if $d(A') = r(B')$.

\begin{lemma}\label{lem:bijection-three} Let $S \subseteq R$ as above.
If $A'$ and $B'$ are completely prime filters in $S$ put $A = (A')^{\uparrow}$ and $B = (B')^{\uparrow}$.
\begin{enumerate} 
\item $d(A') = \mathbf{d}(A)\downarrow$ and $r(A') = \mathbf{r}(A)\downarrow$; in particular, both $d(A')$ and $r(A')$ are completely prime filters in $S$.
\item $\mathbf{d}(A) = d(A')^{\uparrow}$ and $\mathbf{r}(A) = r(A')^{\uparrow}$.
\item $d(A') = r(B')$ if and only if $\mathbf{d}(A) = \mathbf{r}(B)$.
\item If $d(A') = r(B')$, then $(A' \bullet B')^{\uparrow} = A \cdot B$;
it follows that $(A \cdot B)\downarrow \,= A' \bullet B'$.
\item $(\mathsf{C}'(S),\bullet, d,r)$ is a category and isomorphic to the category $\mathsf{C}(R)$.
\end{enumerate}
\end{lemma}
\begin{proof} (1) We shall prove the result for $d(A')$ only, since the proof for  $r(A')$ is similar.
Let $A'$ be a completely prime filter in $S$.
Then $A = (A')^{\uparrow}$ is a completely prime filter in $R$.
Since $A' \subseteq A$ it follows that $(A')^{\ast} subseteq A^{\ast}$ and so $((A')^{\ast})^{\uparrow} \subseteq (A^{\ast})^{\uparrow} = \mathbf{d}(A)$.
Let $b \in  (A^{\ast})^{\uparrow}$. 
Then there is an element $a^{\ast} \in A^{\ast}$ such that $a^{\ast} \leq b$ and $a \in A$.
By assumption, $a = \bigvee_{i \in I} a_{i}$ where the $a_{i}$ are partial isometries.
Observe that $a^{\ast} = \bigvee_{i \in I} a_{i}^{\ast}$.
Since $A$ is a completely prime filter in $R$ it follows that $a_{i} \in A$ for some $i \in I$.
Thus $a_{i}^{\ast} \in (A')^{\ast}$.
It follows that $a^{\ast} \in ((A')^{\ast})^{\uparrow}$ and so $b \in  ((A')^{\ast})^{\uparrow}$.
Thus, $d(A') = \mathbf{d}(A)\downarrow$.
We have already proved that $\mathbf{d}(A)$ is a completely prime filter in $R$,
so it follows that $d(A')$ is a completely prime filter in $S$.

(2) This is immediate by (1).

(3) This is immediate by (1).

(4) Both $A' \bullet B'$ and $A \cdot B$ are deined.
It is clear that $(A' \bullet B')^{\uparrow} \subseteq A \cdot B$.
Let $x \in A \cdot B$.
Then $ab \leq x$ where $a \in A$ and $b \in B$.
By assumption, we can write $a = \bigvee a_{i}$, where the $a_{i}$ are partial isometries,
and $b = \bigvee b_{j}$, where the $b_{j}$ are partial isometries.
Since $A$ and $B$ are each assumed to be completely prime filters in $R$,
it follows that $a_{i} \in A$ and $b_{j} \in B$ for some $i$ and $j$.
Whence, $a_{i} \in A'$ and $b_{j} \in B'$.
Observe that $a_{i}b_{j} \leq x$.
We have therefore proved that $x \in (A'B')^{\uparrow} \subseteq (A' \bullet B')^{\uparrow}$.

(5) This now follows by the results above and Lemma~\ref{lem:bijection-one}. 
\end{proof}

Let $S$ be a complete restriction monoid.
Because of the above results, we shall write $\mathsf{C}(S)$ instead of $\mathsf{C}'(S)$
Thus, $\mathsf{C}(S)$ is an \'etale category whose elements are the completely prime filters of $S$.

Denote by \mbox{\bf CRM} the category of complete restriction monoids and callitic morphisms.
Denote by \mbox{\bf EC} the category of \'etale categories and their morpphisms.
The following theorem now follows by what we have proved in this section,
the results of Section~6 and Theorem~\ref{them:adjunction}.

\begin{theorem}[Adjunction Theorem II]\label{them:adjunction-redux} The functor 
$$\mathsf{C} \colon \mbox{\bf CRM}^{op} \rightarrow \mbox{\bf EC}$$
is right adjoint to the functor 
$$\Omega \colon \mbox{\bf EC} \rightarrow \mbox{\bf CRM}^{op}.$$
\end{theorem}


\end{document}